\documentclass{amsart}
\usepackage{amsmath,amssymb}
\usepackage[dvips]{graphics}
\usepackage[all]{xy}
\newtheorem{Theorem}{Theorem}[section]
\newtheorem{Lemma}[Theorem]{Lemma}
\newtheorem{Proposition}[Theorem]{Proposition}
\newtheorem{Corollary}[Theorem]{Corollary}

\newtheorem{Question}[Theorem]{Question}

\newtheorem{Remark}[Theorem]{Remark}

\def\labelenumi{\rm ({\makebox[0.8em][c]{\theenumi}})}
\def\theenumi{\arabic{enumi}}
\def\labelenumii{\rm {\makebox[0.8em][c]{(\theenumii)}}}
\def\theenumii{\roman{enumii}}

\renewcommand{\thefigure}
{\arabic{section}.\arabic{figure}}
\makeatletter
\@addtoreset{figure}{section}
\def\@thmcountersep{-}
\makeatother

\numberwithin{equation}{section}

\begin{document} 

\title{Symmetries of spatial graphs and Simon invariants}

\author{Ryo Nikkuni}
\address{Department of Mathematics, School of Arts and Sciences, Tokyo Woman's Christian University, 2-6-1 Zempukuji, Suginami-ku, Tokyo 167-8585, Japan}
\email{nick@lab.twcu.ac.jp}
\thanks{The first author was partially supported by Grant-in-Aid for Young Scientists (B) (No. 18740030), Japan Society for the Promotion of Science.}

\author{Kouki Taniyama}
\address{Department of Mathematics, School of Education, Waseda University, Nishi-Waseda 1-6-1, Shinjuku-ku, Tokyo, 169-8050, Japan}
\email{taniyama@waseda.jp}
\thanks{The second author was partially supported by Grant-in-Aid for Scientific Research (C) (No. 18540101), Japan Society for the Promotion of Science.}


\date{}

\dedicatory{}


\begin{abstract}
An ordered and oriented $2$-component link $L$ in the $3$-sphere is said to be achiral if it is ambient isotopic to its mirror image ignoring the orientation and ordering of the components. Kirk-Livingston showed that if $L$ is achiral then the linking number of $L$ is not congruent to $2$ modulo $4$. In this paper we study orientation-preserving or reversing symmetries of $2$-component links, spatial complete graphs on $5$ vertices and spatial complete bipartite graphs on $3+3$ vertices in detail, and determine the necessary conditions on linking numbers and Simon invariants for such links and spatial graphs to be symmetric.
\end{abstract}

\maketitle

\section{Introduction} 

Throughout this paper we work in the piecewise linear category. Let $L=J_{1}\cup J_{2}$ be an ordered and oriented $2$-component link in the unit $3$-sphere ${\mathbb S}^{3}$. Unless otherwise stated, the links in this paper will be ordered and oriented. A link $L$ is said to be {\it component preserving achiral} (CPA) if there exists an orientation-reversing self-homeomorphism $\varphi$ of ${\mathbb S}^{3}$ such that $\varphi(J_{1})=J_{1}$ and $\varphi(J_{2})=J_{2}$, and {\it component switching achiral} (CSA) if there exists an orientation-reversing self-homeomorphism $\varphi$ of ${\mathbb S}^{3}$ such that $\varphi(J_{1})=J_{2}$ and $\varphi(J_{2})=J_{1}$ \cite{kidwell06}. If $L$ is either CPA or CSA, then $L$ is said to be {\it achiral}. Note that $L$ may be both CPA and CSA (a trivial link, for example). The following was shown by Kirk-Livingston. 

\begin{Theorem}\label{mod4} 
{\rm (\cite[\sc 6.1 Corollary]{kirk-livingston97})} 
If $L$ is achiral then ${\rm lk}(L)$ is not congruent to $2$ modulo $4$, where ${\rm lk}$ denotes the linking number. 
\end{Theorem}

See also \cite[Theorem 5.1]{livingston03} for an elementary proof of Theorem \ref{mod4}.  
Note that for any odd integer $n$ there exists a $2$-component link of linking number $n$ which is both CPA and CSA; see Fig. \ref{CPACSA} (cf. \cite[\S 5]{livingston03}). 
In \cite{livingston03}, Livingston gave an example of a CSA link with linking number $4$ (a cabling of the Hopf link) and stated open problems: to find an achiral link of linking number $4m$ for any integer $m$, and to find a CPA link of linking number $4m$ for any integer $m$''. For the latter, Kidwell showed the following.

\begin{figure}[htbp]
      \begin{center}
\scalebox{0.325}{\includegraphics*{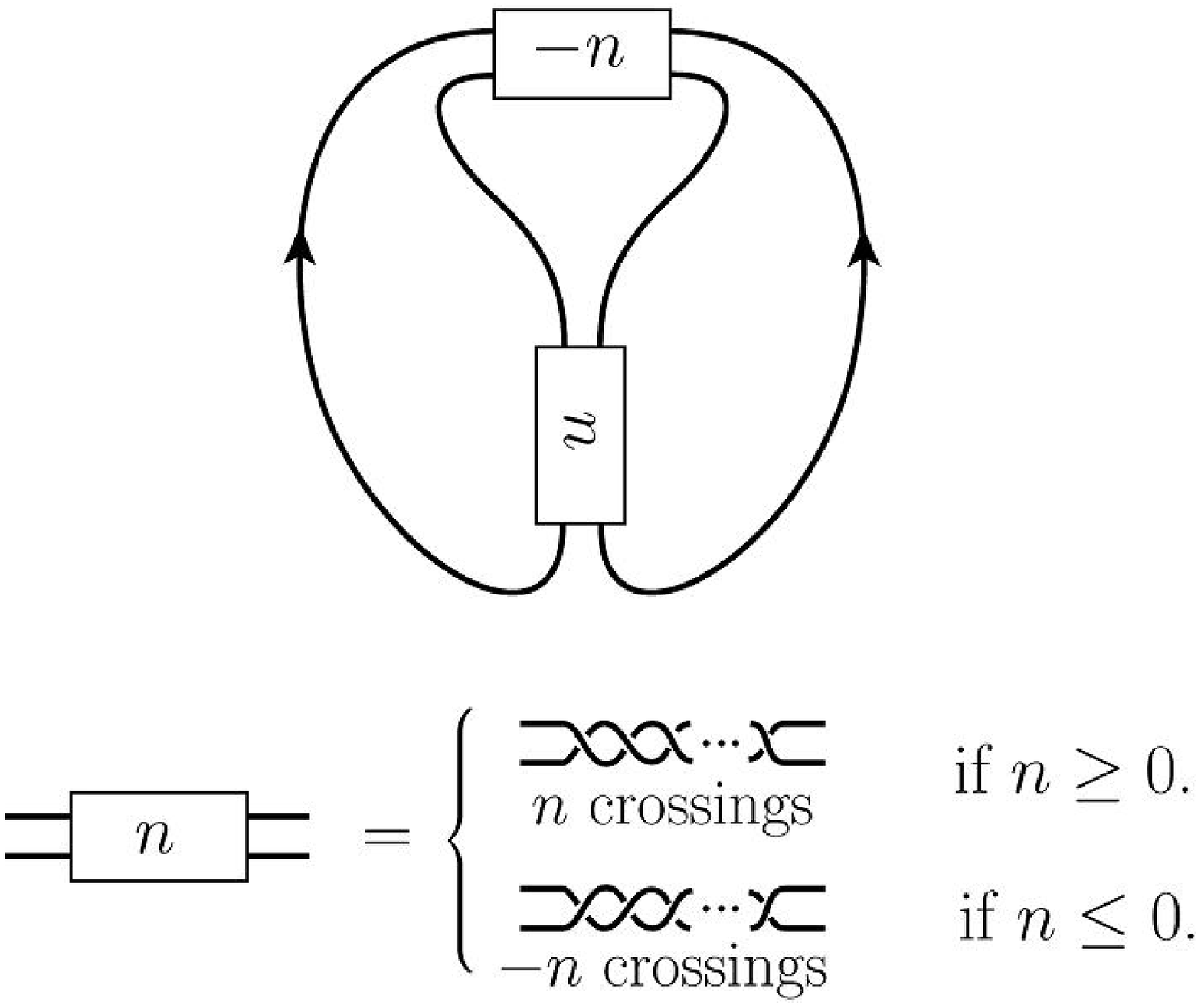}}
      \end{center}
   \caption{}
  \label{CPACSA}
\end{figure} 

\begin{Theorem}\label{CPA_oddlk} 
{\rm (\cite[Theorem 4]{kidwell06})} 
A $2$-component link of nonzero even linking number cannot be CPA. 
\end{Theorem}

Kidwell also gave an example of a CSA link of linking number $4m$ for any odd integer $m$ \cite[\S 3]{kidwell06}. But as far as the authors know, a CSA link of linking number $4m$ for any nonzero even integer $m$ has not been exhibited yet. 

To give a complete answer to Livingston's problem, we will present a new family of achiral links. For an integer $m$ and $\varepsilon_{1},\varepsilon_{2}\in \left\{-1,1\right\}$, let $L(m,\varepsilon_{1},\varepsilon_{2})$ be a $2$-component link of linking number $4m+(\varepsilon_{1}+\varepsilon_{2})/2$ as illustrated in Fig. \ref{achiral_link}. Note that $L(0,\varepsilon,-\varepsilon)$ is trivial for $\varepsilon=\pm 1$. Therefore $L(0,\varepsilon,-\varepsilon)$ is both a CPA and CSA link of linking number $0$. Moreover we have the following. 

\begin{figure}[htbp]
      \begin{center}
\scalebox{0.375}{\includegraphics*{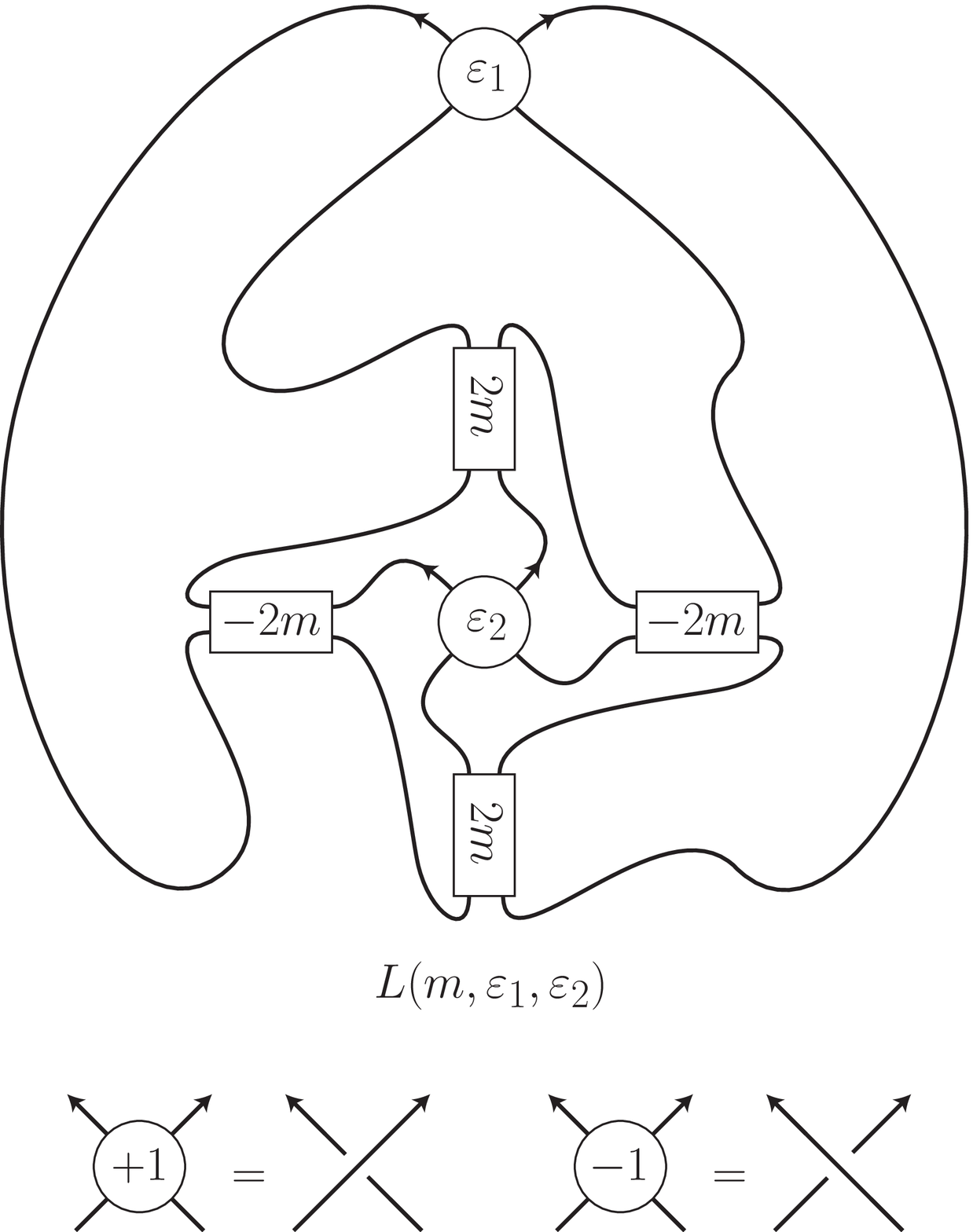}}
      \end{center}
   \caption{}
  \label{achiral_link}
\end{figure} 

\begin{Theorem}\label{newexample}
{\rm (1)} For any integer $m$ and $\varepsilon_{1},\varepsilon_{2}\in \left\{-1,1\right\}$, $L(m,\varepsilon_{1},\varepsilon_{2})$ is CSA. 

\noindent
{\rm (2)} For any integer $m$ and $\varepsilon=\pm 1$, $L(m,\varepsilon,\varepsilon)$ is CPA. 
\end{Theorem}

\begin{proof} In Fig. 1.2, we may suppose that a regular diagram of $L(m,\varepsilon_1,\varepsilon_2)$ on the standard 2-sphere ${\mathbb S}^{2}$ in ${\mathbb S}^{3}$ is such that the crossing of sign $\varepsilon_1$ is on the north pole and that of sign $\varepsilon_2$ is on the south pole. 

\noindent
(1) A $\pi/2$ rotation of ${\mathbb S}^{2}$ around the earth's axis maps $L(m,\varepsilon_1,\varepsilon_2)$ onto its mirror image. Thus we have the assertion. 

\noindent
(2) The $\pi$ rotation of ${\mathbb S}^{2}$ around an axis through the equator interchanges the components of $L(m,\varepsilon,\varepsilon)$. Then by composing it with a $\pi/2$ rotation of ${\mathbb S}^{2}$ around the earth's axis, we see that $L(m,\varepsilon,\varepsilon)$ is mapped onto its mirror image preserving the components. Thus we have the result. 
\end{proof}

The following corollary shows that Theorems \ref{mod4} and \ref{CPA_oddlk} give the best possible necessary conditions on the linking number for a $2$-component link to be CPA or to be CSA. 

\begin{Corollary}\label{newexample_cor}
{\rm (1)} For any integer $m$ and $\varepsilon=\pm 1$, $L(m,\varepsilon,\varepsilon)$ is both a CPA and CSA link of linking number $4m+\varepsilon$. 

\noindent
{\rm (2)} For any integer $m$ and $\varepsilon=\pm 1$, 
$L(m,\varepsilon,-\varepsilon)$ is a CSA link of linking number $4m$. 
\end{Corollary}

\begin{Remark}\label{sym_remark}
{\rm A further step in the investigation of the achirality of $2$-component links is to take the invertibility of each component into account, as follows. For a $2$-component link $L=J_{1}\cup J_{2}$, assume that there exists a self-homeomorphism $\varphi$ of ${\mathbb S}^{3}$ such that $\varphi(L)=L$. Then the candidates of $(\varphi(J_{1}),\varphi(J_{2}))$ are (1) $(J_{1},J_{2})$, (2) $(J_{2},J_{1})$, (3) $(-J_{1},-J_{2})$, (4) $(-J_{2},-J_{1})$, (5) $(-J_{1},J_{2})$, (6) $(J_{1},-J_{2})$, (7) $(-J_{2},J_{1})$, (8) $(J_{2},-J_{1})$. When ${\rm lk}(L)\neq 0$, by observing the change of the sign of ${\rm lk}(L)$ it can be seen that the first four cases may occur only if $\varphi$ is orientation-preserving, and the last four only if $\varphi$ is orientation-reversing, where $L$ is CPA in case of (5) or (6), and CSA in case of (7) or (8). The first four cases can be realized by a $(2,2n)$-torus link simultaneously for any integer $n$, so these symmetries do not depend on the linking number. On the other hand, since $L(m,\varepsilon_{1},\varepsilon_{2})$ is invertible, only Theorems \ref{mod4} and \ref{CPA_oddlk} are restrictions of the linking number for the latter four cases.
}
\end{Remark}

Next, let us consider finite graphs which are embedded in ${\mathbb S}^{3}$. An embedding $f$ of a finite graph $G$ into ${\mathbb S}^{3}$ is called a {\it spatial embedding of $G$} or simply a {\it spatial graph}. Two spatial embeddings $f$ and $g$ of $G$ are said to be {\it ambient isotopic} if there exists an orientation-preserving self-homeomorphism $\varphi$ of ${\mathbb S}^{3}$ such that $\varphi\circ f=g$. In \cite{taniyama94b}, the second author introduced the notion of {\it (spatial graph-)homology} as a fundamental equivalence relation on spatial graphs which is weaker than ambient isotopy. It is known that a homology is generated by {\it Delta moves} and ambient isotopies \cite[Theorem 1.3]{motohashi-taniyama97}, where a Delta move is a local deformation of a spatial graph as illustrated in Fig. \ref{delta}. Linking numbers of constituent $2$-component links are typical homological invariants of spatial graphs. In particular, two $k$-component links are homologous if and only if they have the same pairwise linking numbers \cite[Theorem 1.1]{murakami-nakanishi89}. 

\begin{figure}[htbp]
      \begin{center}
\scalebox{0.45}{\includegraphics*{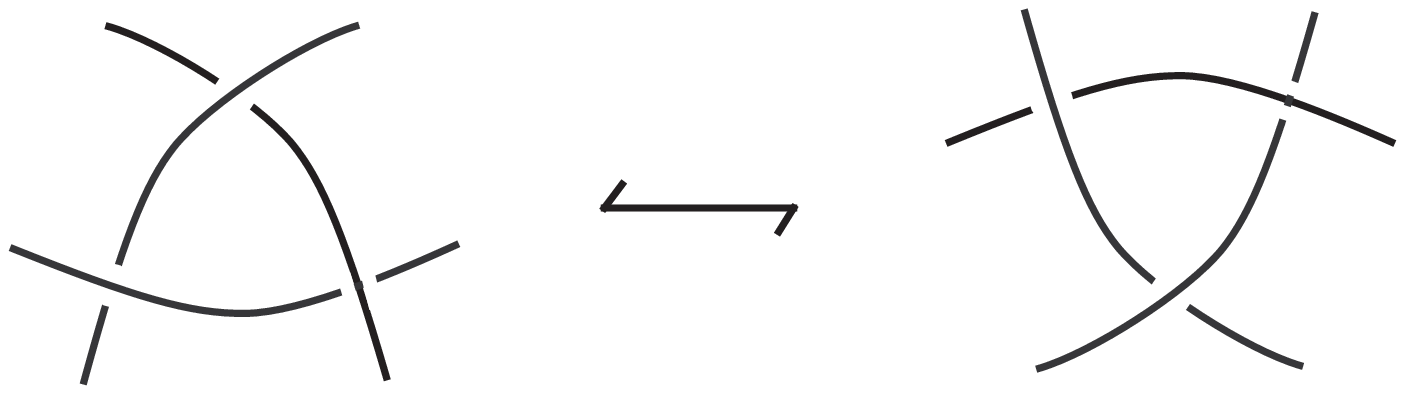}}
      \end{center}
   \caption{}
  \label{delta}
\end{figure} 

On the other hand, let $K_{5}$ and $K_{3,3}$ be a {\it complete graph} on five vertices and a {\it complete bipartite graph} on $3+3$ vertices respectively, known as obstructions in Kuratowski's graph planarity criterion \cite{kuratowski30}. For a spatial embedding $f$ of $K_{5}$ or $K_{3,3}$, the {\it Simon invariant} ${\mathcal L}(f)$ is defined \cite[$\S$4]{taniyama94b}, which is an odd integer valued homological invariant calculated from the regular diagram of $f$, like the linking number. We give a precise definition of ${\mathcal L}(f)$ in the next section. It is known that two spatial embeddings of a graph are homologous if and only if they have the same {\it Wu invariant} \cite{taniyama95}, or equivalently, their corresponding constituent $2$-component links have the same linking number and their corresponding spatial subgraphs which are homeomorphic to $K_{5}$ or $K_{3,3}$ have the same Simon invariant \cite{shinjo-taniyama03}. Thus, linking numbers and Simon invariants play a primary role in classifying spatial graphs up to homology. We remark here that both the linking number and the Simon invariant come from the Wu invariant as special cases, see \cite[$\S$2]{taniyama95} for details. 

Our main purpose in this paper is to reveal the relationship between orientation-preserving or reversing symmetries of spatial embeddings of $K_{5}$ and $K_{3,3}$ and their Simon invariants in addition to the case of $2$-component links and their linking numbers. A study of symmetries of spatial graphs is not only a fundamental problem in low dimensional topology as a generalization of achirality of knots and links but also an important research theme from a standpoint of application to macromolecular chemistry, which is called molecular topology \cite{walba83}, \cite{simon87}, \cite{flapan89}, \cite{flapan00}. Actually, both $K_{5}$ and $K_{3,3}$ have been realized chemically as underlying molecule graphs for some chemical compounds \cite{simmons-maggio81}, \cite{paquette-vazeux81}, \cite{wrh82}, \cite{kuck-schuster88}. We refer the reader to  \cite{simon86} for a pioneer work on symmetries of spatial embeddings of $K_{5}$ and $K_{3,3}$.

From now on we assume that $G$ is $K_{5}$ or $K_{3,3}$. Let ${\rm Aut}(G)$ be the {\it automorphism group} of $G$. We regard each automorphism of $G$ as a self-homeomorphism of $G$. For a spatial embedding $f$ of $G$ and an automorphism $\sigma$ on $G$, we say that $f$ is {\it $\sigma$-symmetric} (resp. {\it rigidly $\sigma$-symmetric}) if there exists a self-homeomorphism (resp. periodic self-homeomorphism) $\varphi$ of ${\mathbb S}^{3}$ such that $f\circ \sigma=\varphi\circ f$. In particular, if $\varphi$ is orientation-reversing then we say that $f$ is {\it $\sigma$-achiral} (resp. {\it rigidly $\sigma$-achiral}). In the first author's preliminary report \cite{nick07}, it was shown that for any odd integer $n$ there exist an automorphism $\sigma$ on $G$ and a $\sigma$-achiral spatial embedding $f$ of $G$ such that ${\mathcal L}(f)=n$. Then it is natural to ask the following question. 

\begin{Question}\label{question} 
For an automorphism $\sigma$ on $G$ and an odd integer $n$, does there exist a $\sigma$-symmetric spatial embedding $f$ of $G$ such that ${\mathcal L}(f)=n$?
\end{Question}

For an automorphism $\sigma$ of $G$ and an odd integer $n$, we say that the pair $(\sigma,n)$ is {\it realizable} if there exists a $\sigma$-symmetric spatial embedding $f$ of $G$ such that ${\mathcal L}(f)=n$. Then the following proposition holds. 

\begin{Proposition}\label{conjugate1}
Let $\sigma$ and $\tau$ be two automorphisms of $G$ which are conjugate in ${\rm Aut}(G)$, and $n$ an odd integer. Then $(\sigma,n)$ is realizable if and only if $(\tau,n)$ is realizable. 
\end{Proposition}

Therefore it is sufficient to consider only conjugacy classes in ${\rm Aut}(G)$. Note that we may identify ${\rm Aut}(G)$ with a subgroup of the symmetric group ${\mathfrak S}_{m}$ of degree $m$, where $m$ is the number of vertices of $G$. If $G$ is $K_{3,3}$, we assume that the vertices corresponding to $1$, $2$ and $3$ are adjacent to the vertices corresponding to $4$, $5$ and $6$. Then we have the following.

\begin{Proposition}\label{conjugate2}
{\rm (1)} Representatives for all conjugacy classes in ${\rm Aut}(K_{5})$ are 
\begin{eqnarray*}
{\rm id},\ (1\ 2),\ (1\ 2\ 3),\ (1\ 2\ 3\ 4),\ (1\ 2\ 3\ 4\ 5),\ (1\ 2)(3\ 4),\   (1\ 2)(3\ 4\ 5).
\end{eqnarray*}

\noindent
{\rm (2)} Representatives for all conjugacy classes in ${\rm Aut}(K_{3,3})$ are 
\begin{eqnarray*}
&&{\rm id},\ (1\ 2),\ (1\ 2\ 3),\ (1\ 4\ 2\ 5\ 3\ 6),\ (1\ 2)(4\ 5),\ (1\ 2)(4\ 5\ 6),\\
&& (1\ 2\ 3)(4\ 5\ 6),\ (1\ 4\ 2\ 5)(3\ 6),\ (1\ 4)(2\ 5)(3\ 6). 
\end{eqnarray*}
\end{Proposition}

Now we state our main theorems. 

\begin{Theorem}\label{K5}
Let $G$ be $K_{5}$ and $n$ an odd integer. Then we have the following. 

\noindent
{\rm (1)} $((1\ 2\ 3),n)$ is realizable if and only if $n$ is congruent to $\pm 1$ modulo $6$. 

\noindent
{\rm (2)} $((1\ 2\ 3\ 4\ 5),n)$ is realizable. 

\noindent
{\rm (3)} $((1\ 2)(3\ 4),n)$ is realizable. 

\noindent
{\rm (4)} $((1\ 2\ 3\ 4),n)$ is realizable. 

\noindent
{\rm (5)} $((1\ 2),n)$ is realizable if and only if $n=\pm 1$. 

\noindent
{\rm (6)} $((1\ 2)(3\ 4\ 5),n)$ is realizable if and only if $n=\pm 1$. 
\end{Theorem}

\begin{Theorem}\label{K33}
Let $G$ be $K_{3,3}$ and $n$ an odd integer. Then we have the following. 

\noindent
{\rm (1)} $((1\ 2\ 3),n)$ is realizable if and only if $n$ is congruent to $\pm 1$ modulo $6$. 

\noindent
{\rm (2)} $((1\ 2)(4\ 5),n)$ is realizable. 

\noindent
{\rm (3)} $((1\ 2\ 3)(4\ 5\ 6),n)$ is realizable. 

\noindent
{\rm (4)} $((1\ 4)(2\ 5)(3\ 6),n)$ is realizable. 

\noindent
{\rm (5)} $((1\ 4\ 2\ 5\ 3\ 6),n)$ is realizable. 

\noindent
{\rm (6)} $((1\ 4\ 2\ 5)(3\ 6),n)$ is realizable. 

\noindent
{\rm (7)} $((1\ 2),n)$ is realizable if and only if $n=\pm 1$. 

\noindent
{\rm (8)} $((1\ 2)(4\ 5\ 6),n)$ is realizable if and only if $n=\pm 1$. 
\end{Theorem}

Furthermore, each realizable pair $(\sigma,n)$ may be realized by a rigidly $\sigma$-symmetric spatial embedding of $G$. By combining Theorems \ref{K5} and \ref{K33} with Propositions \ref{conjugate1} and \ref{conjugate2}, we can give a complete answer to Question \ref{question}. We can also determine when $(\sigma,n)$ is realized by a $\sigma$-achiral spatial embedding of $G$ as follows. 

\begin{Theorem}\label{achiral}
{\rm (1)} Let $\sigma$ be an automorphism of $K_{5}$ and $n$ an odd integer. Suppose that $(\sigma,n)$ is realizable. Then $(\sigma,n)$ is realized by a $\sigma$-achiral spatial embedding of $K_{5}$ if and only if $\sigma$ is conjugate to $(1\ 2\ 3\ 4),\ (1\ 2)$ or $(1\ 2)(3\ 4\ 5)$. 

\noindent
{\rm (2)} Let $\sigma$ be an automorphism of $K_{3,3}$ and $n$ an odd integer. Suppose that $(\sigma,n)$ is realizable. Then $(\sigma,n)$ is realized by a $\sigma$-achiral spatial embedding of $K_{3,3}$ if and only if $\sigma$ is conjugate to $(1\ 4\ 2\ 5)(3\ 6),\ (1\ 2)$ or $(1\ 2)(4\ 5\ 6)$.
\end{Theorem}

Actually we will show that a realizable pair $(\sigma,n)$ in Theorem \ref{K5} (1), (2), (3) and Theorem \ref{K33} (1), (2), (3), (4), (5) can only be realized by an orientation-preserving self-homeomorphism of ${\mathbb S}^{3}$, and a realizable pair $(\sigma,n)$ in Theorem \ref{K5} (4), (5), (6) and Theorem \ref{K33} (6), (7), (8) can only be realized by an orientation-reversing self-homeomorphism of ${\mathbb S}^{3}$. 

In the next section we give a precise definition of the Simon invariant. In section $3$, we prove some propositions including Propositions \ref{conjugate1} and \ref{conjugate2} that are needed later. The proofs of Theorems \ref{K5}, \ref{K33} and \ref{achiral} are given in section $4$. 

\section{Simon invariant} 

Let $G$ be $K_{5}$ or $K_{3,3}$ where the vertices have a fixed numbering. In this section we review the Simon invariant of spatial embeddings of $G$. We give an orientation to each edge of $G$ as illustrated in Fig. \ref{K5K33label}, according to the numbering of the vertices. For an unordered pair $x,y$ of disjoint edges $K_{5}$, we define the sign $\varepsilon(x,y)$ by $\varepsilon(e_{i},e_{j})=1$, $\varepsilon(d_{k},d_{l})=-1$ and $\varepsilon(e_{i},d_{k})=-1$. For an unordered pair $x,y$ of disjoint edges of $K_{3,3}$, we also define the sign $\varepsilon(x,y)$ by $\varepsilon(c_{i},c_{j})=1$, $\varepsilon(b_{k},b_{l})=1$ and 
$\varepsilon(c_{i},b_{k})=1$ if $c_{i}$ and $b_{k}$ are parallel in Fig. \ref{K5K33label}, and $-1$ if they are anti-parallel. For a spatial embedding $f$ of $G$, we fix a regular diagram of $f$ and denote the sum of the signs of the crossing points between $f(x)$ and $f(y)$ by $l(f(x),f(y))$, where $x,y$ is an unordered pair of disjoint edges of $G$. Now we define an integer ${\mathcal L}(f)$ by 
\begin{eqnarray*}
{\mathcal L}(f)=\sum_{(x,y)}\varepsilon(x,y)l(f(x),f(y)), 
\end{eqnarray*}
where the summation is taken over all unordered pairs of disjoint edges of $G$. This integer ${\mathcal L}(f)$ is called the {\it Simon invariant} of $f$. Actually, this is an odd integer valued ambient isotopy invariant \cite[\sc Theorem 4.1]{taniyama94b} up to the numbering of the vertices. In paricular, for a different numbering of the vertices, the value of the Simon invariant may be different. The Simon invariant can also be described from a cohomological viewpoint as follows. See \cite[Example 2,4, Example 2.5]{taniyama95} for details. Let $C_{2}(X)$ be the {\it configuration space} of ordered two points of a topological space $X$, namely
\begin{eqnarray*}
C_{2}(X) = \left\{(x,y)\in X\times X~|~x\neq y\right\}. 
\end{eqnarray*}
Let $\iota$ be an involution on $C_{2}(X)$ defined by $\iota(x,y)=(y,x)$. Then we call the integral cohomology group of ${\rm Ker}(1 + \iota_{\sharp})$ the {\it skew-symmetric integral cohomology group} of the pair $(C_{2}(X),\iota)$ and denote it by $H^{*}(C_{2}(X),\iota)$. It is known that $H^{2}(C_{2}({\mathbb R}^{3}),\iota)\cong {\mathbb Z}$ \cite{wu60} and $H^{2}(C_{2}(G),\iota)\cong {\mathbb Z}$. We denote a generator of $H^{2}(C_{2}({\mathbb R}^{3}),\iota)$ by $\Sigma$. Let $f:G\to {\mathbb S}^{3}\setminus\left\{(0,0,0,1)\right\}$ be a spatial embedding of $G$, regarded as an embedding into ${\mathbb R}^{3}$. This embedding $f$ naturally induces an equivariant embedding $f\times f:C_{2}(G)\to C_{2}({\mathbb R}^{3})$ with respect to the action $\iota$ and therefore induces a homomorphism 
\begin{eqnarray*}
(f\times f)^{*}:H^{2}(C_{2}({\mathbb R}^{3}),\iota)
\longrightarrow H^{2}(C_{2}(G),\iota). 
\end{eqnarray*}
Then the Simon invariant ${\mathcal L}(f)$ coincides with $(f\times f)^{*}(\Sigma)$ up to sign. Thus the absolute value $|{\mathcal L}(f)|$ is an ambient isotopy invariant independent of the numbering of the vertices.

\begin{figure}[htbp]
      \begin{center}
\scalebox{0.375}{\includegraphics*{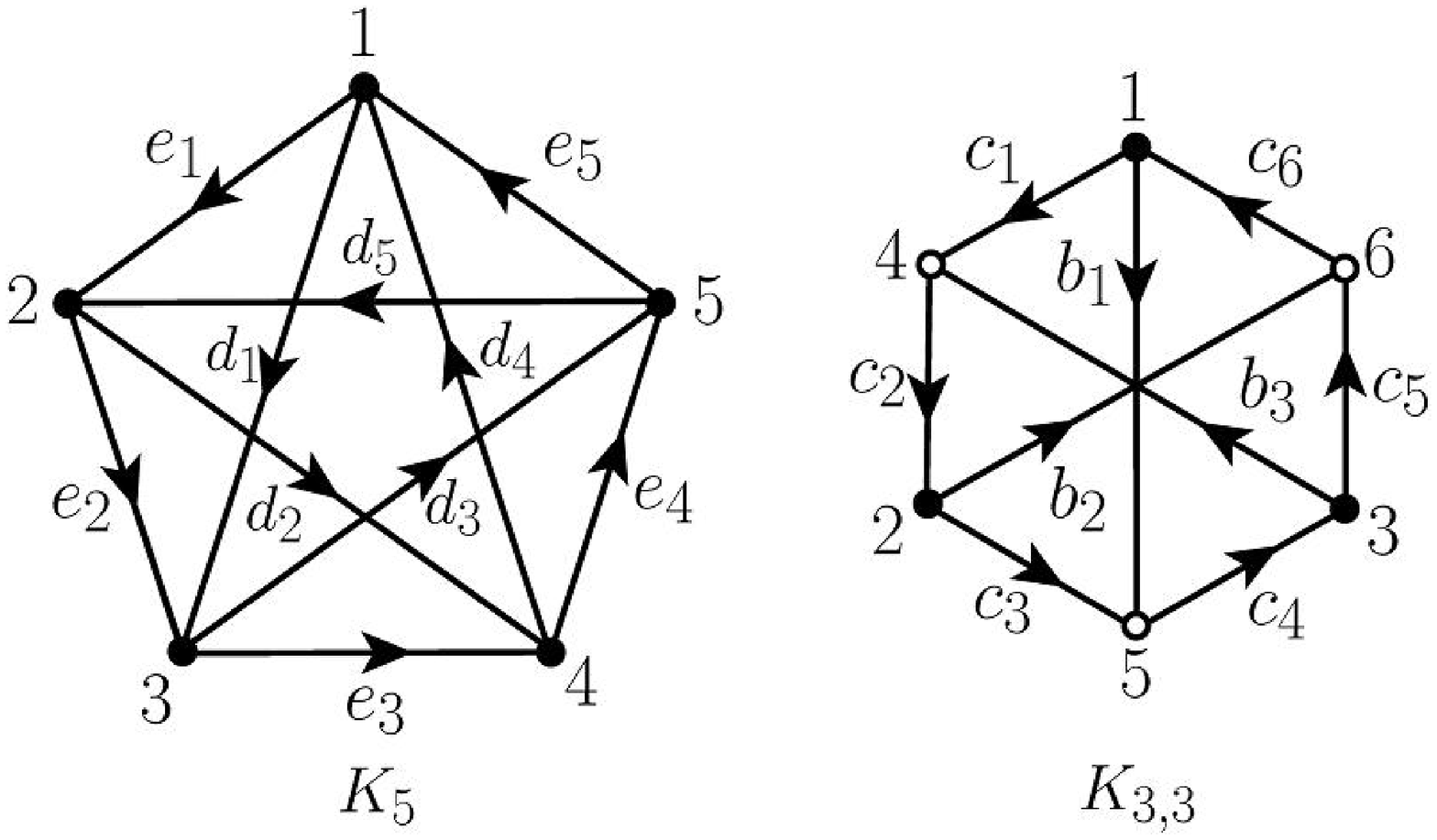}}
      \end{center}
   \caption{}
  \label{K5K33label}
\end{figure} 

The Simon invariant ${\mathcal L}(f)$ is also closely related to the constituent knots of $f$. Let $\Gamma(G)$ be the set of all cycles of $G$, where a {\it cycle} is a subgraph of $G$ which is homeomorphic to a circle. We say that a cycle is a {\it $k$-cycle} if it contains exactly $k$ edges. Let $\omega:\Gamma(G)\to {\mathbb Z}$ be the map defined by  
\begin{eqnarray*}
\omega(\gamma)=\left\{
       \begin{array}{@{\,}ll}
       1 & \mbox{(if $\gamma$ is a $5$-cycle)}\\
       -1 & \mbox{(if $\gamma$ is a $4$-cycle)}\\
       0 & \mbox{(if $\gamma$ is a $3$-cycle)}
       \end{array}
     \right.
\end{eqnarray*}
for $G=K_{5}$, and 
\begin{eqnarray*}
\omega(\gamma)=\left\{
       \begin{array}{@{\,}ll}
       1 & \mbox{(if $\gamma$ is a $6$-cycle)}\\
       -1 & \mbox{(if $\gamma$ is a $4$-cycle)}\\
       \end{array}
     \right.
\end{eqnarray*}
for $G=K_{3,3}$. For a spatial embedding $f$ of $G$, we define an integer $\alpha_{\omega}(f)$ by 
\begin{eqnarray*}
\alpha_{\omega}(f)=\sum_{\gamma\in \Gamma(G)}\omega(\gamma)a_{2}(f(\gamma)), 
\end{eqnarray*}
where $a_{2}(J)$ is the second coefficient of the {\it Conway polynomial} of a knot $J$. This integer $\alpha_{\omega}(f)$ is called the {\it $\alpha$-invariant} of $f$ \cite{taniyama94a}. In \cite{motohashi-taniyama97}, Motohashi and the second author showed that if ${\mathcal L}(f)=2j-1$ then 
\begin{eqnarray*}
\alpha_{\omega}(f)=\frac{j(j-1)}{2}. 
\end{eqnarray*}
This implies that 
\begin{eqnarray}
\alpha_{\omega}(f)=\frac{{{\mathcal L}(f)}^{2}-1}{8}\label{simon_alpha}
\end{eqnarray}
for a spatial embedding $f$ of $G$. Then we have the following. 

\begin{Lemma}\label{mod6}
Let $G=K_{5}$ or $K_{3,3}$, and let $f$ be a spatial embedding of $G$. Then $\alpha_{\omega}(f)$ is a multiple of $3$ if and only if ${\mathcal L}(f)\equiv \pm 1\pmod{6}$. 
\end{Lemma}

\begin{proof}
It is easy to check that $m \equiv \pm 1\pmod{6}$ if and only if $m^{2} \equiv \pm 1\pmod{6}$. Therefore ${\mathcal L}(f)\equiv \pm 1\pmod{6}$ if and only if ${\mathcal L}(f)^2-1$ is a multiple of $6$. By (\ref{simon_alpha}) we have ${\mathcal L}(f)^2-1=8\alpha_\omega(f)$. Since $8 \alpha_\omega(f)$ is a multiple of $6$ if and only if $\alpha_\omega(f)$ is a multiple of $3$ we have the desired conclusion. 
\end{proof}

\section{Conjugacy classes in ${\rm Aut}(G)$ and Simon invariants} 

In this section we discuss the relationship between conjugacy classes in ${\rm Aut}(G)$ and Simon invariants. 

\begin{Lemma}\label{mirror}
Let $G=K_{5}$ or $K_{3,3}$. For an automorphism $\sigma$ of $G$ and an odd integer $n$, if $(\sigma,n)$ is realizable then $(\sigma,-n)$ is realizable. 
\end{Lemma}

\begin{proof}
Let $\rho$ be an orientation-reversing self-homeomorphism of ${\mathbb S}^{3}$ defined by $\rho(x_{1},x_{2},x_{3},x_{4})=(x_{1},x_{2},x_{3},-x_{4})$, and $f$ a spatial embedding of $G$. We call $\rho\circ f$ the {\it mirror image embedding} of $f$ and denote it by $f!$. If $(\sigma,n)$ is realized by $f$, there exists a self-homeomorphism $\varphi$ of ${\mathbb S}^{3}$ such that $f\circ \sigma=\varphi\circ f$. Then the following diagram is commutative. 
\[\xymatrix{
G \ar[d]_{\sigma} \ar[r]^{f} & 
{\mathbb S}^{3} \ar[d]_{\varphi} \ar[r]^{\rho}  & 
{\mathbb S}^{3} \ar[d]^{\rho\circ\varphi\circ\rho^{-1}} & \\
G \ar[r]_{f} & 
{\mathbb S}^{3} \ar[r]_{\rho} &
{\mathbb S}^{3} & 
}
\]
Thus, $f!$ is $\sigma$-symmetric. By the definition of the Simon invariant, we can see that ${\mathcal L}(f!)=-{\mathcal L}(f)=-n$. Therefore $(\sigma,-n)$ is realizable. 
\end{proof}

\begin{Lemma}\label{auto_simon}
Let $G=K_{5}$ or $K_{3,3}$, and $f$ a spatial embedding of $G$. Then ${\mathcal L}(f\circ \xi)=\pm {\mathcal L}(f)$ for any automorphism $\xi$ of $G$. 
\end{Lemma}

\begin{proof}
The automorphism $\xi$ naturally induces an equivariant homeomorphism $\xi\times\xi:C_{2}(G)\to C_{2}(G)$ with respect to the action $\iota$ and therefore induces an isomorphism 
\begin{eqnarray*}
(\xi\times \xi)^{*}:H^{2}(C_{2}(G),\iota)\stackrel{\cong}{\to} H^{2}(C_{2}(G),\iota).
\end{eqnarray*}
Then we have  
\begin{eqnarray*}
{\mathcal L}(f\circ \xi)=((f\circ\xi)\times (f\circ\xi))^{*}(\Sigma)=(\xi\times \xi)^{*}((f\times f)^{*}(\Sigma))=(\xi\times \xi)^{*}({\mathcal L}(f)), 
\end{eqnarray*}
where $\Sigma$ is a suitable generator of $H^{2}(C_{2}({\mathbb R}^{3}),\iota)$; see the following commutative diagram: 
\[\xymatrix{
& H^{2}(C_{2}(G),\iota)   & \\
& H^{2}(C_{2}(G),\iota) \ar[u]^{(\xi\times \xi)^{*}}_{\cong} & \ar[l]^{(f\times f)^{*}} \ar[ul]_{{~}\ {~}\ ((f\circ\xi)\times (f\circ\xi))^{*}} H^{2}(C_{2}({\mathbb R}^{3}),\iota) & \\
}\]
Since $H^{2}(C_{2}(G),\iota)\cong {\mathbb Z}$, we have the desired conclusion.
\end{proof}

\begin{proof}[Proof of Proposition \ref{conjugate1}.] 
Assume that there exist a spatial embedding $f$ of $G$, a self-homeomorphism $\varphi$ of ${\mathbb S}^{3}$ and an automorphism $\xi$ of $G$ such that $f\circ \sigma=\varphi\circ f$ and $\tau=\xi^{-1}\sigma \xi$. Then the following diagram is commutative:  
\[\xymatrix{
G \ar[d]_{\tau} \ar[r]^{\xi}  & 
G \ar[d]_{\sigma} \ar[r]^{f}  & 
{\mathbb S}^{3} \ar[d]^{\varphi} & \\
G \ar[r]_{\xi} & 
G \ar[r]_{f} &
{\mathbb S}^{3} & 
}
\]
Hence, $f\circ \xi$ is $\tau$-symmetric. Thus, by Lemma \ref{auto_simon}, ${\mathcal L}(f\circ \xi)=\pm {\mathcal L}(f)=\pm n$. If ${\mathcal L}(f\circ \xi)=n$, then $(\tau,n)$ is realizable. If ${\mathcal L}(f\circ \xi)=-n$, then $(\tau,n)$ is also realizable by Lemma \ref{mirror}. 
\end{proof}

\begin{proof}[Proof of Proposition \ref{conjugate2}.] 
(1) We may identify ${\rm Aut}(K_{5})$ with the symmetric group of degree $5$. It is not hard to see that all conjugacy classes in ${\rm Aut}(K_{5})$ are classified as follows: 

\noindent
(i) ${\rm id}$, 

\noindent
(ii) $(i\ j)$ for $\{i,j\}\subset \{1,2,3,4,5\}$, 

\noindent
(iii) $(i\ j\ k)$ for $\{i,j,k\}\subset \{1,2,3,4,5\}$, 

\noindent
(iv) $(i\ j\ k\ l)$ for $\{i,j,k,l\}\subset \{1,2,3,4,5\}$, 

\noindent
(v) $(i\ j\ k\ l\ m)$ for $\{i,j,k,l,m\}=\{1,2,3,4,5\}$, 

\noindent
(vi) $(i\ j)(k\ l)$ for $\{i,j\}\subset \{1,2,3,4,5\}$ and $\{k,l\}\subset \{1,2,3,4,5\}\setminus \{i,j\}$, 

\noindent
(vii) $(i\ j)(k\ l\ m)$ for $\{i,j\}\subset \{1,2,3,4,5\}$ and $\{k,l,m\}=\{1,2,3,4,5\}\setminus \{i,j\}$. 

\noindent
Hence we have the desired representatives. We omit the details. 

\noindent
(2) In a similar way, we may identify ${\rm Aut}(K_{3,3})$ with a subgroup of ${\mathfrak S}_{6}$ of order $10$ (Note that ${\rm Aut}(K_{3,3})$ has a structure of the {\it wreath product} ${\mathfrak S}_{2}[{\mathfrak S}_{3}]$). It is not hard to see that all conjugacy classes in ${\rm Aut}(K_{3,3})$ are classified as follows: 

\noindent
(i) ${\rm id}$, 

\noindent
(ii) $(i\ j)$ for $\{i,j\}\subset\{1,2,3\}$ or $\{i,j\}\subset\{4,5,6\}$,

\noindent
(iii) $(i\ j\ k)$ for $\{i,j,k\}=\{1,2,3\}$ or $\{4,5,6\}$, 

\noindent
(iv) $(i\ l\ j\ m\ k\ n)$ for $\{i,j,k\}=\{1,2,3\}$ and $\{l,m,n\}=\{4,5,6\}$, 

\noindent
(v) $(i\ j)(l\ m)$ for $\{i,j\}\subset \{1,2,3\}$ and $\{l,m\}\subset \{4,5,6\}$, 

\noindent
(vi) $(i\ j)(l\ m\ n)$ for $\{i,j\}\subset \{1,2,3\}$ and $\{l,m,n\}=\{4,5,6\}$, or $\{i,j\}\subset \{4,5,6\}$ and $\{l,m,n\}=\{1,2,3\}$, 

\noindent
(vii) $(i\ j\ k)(l\ m\ n)$ for $\{i,j,k\}=\{1,2,3\}$ and $\{l,m,n\}=\{4,5,6\}$,

\noindent
(viii) $(i\ l\ j\ m)(k\ n)$ for $\{i,j\}\subset \{1,2,3\}$, $\{l,m\}\subset \{4,5,6\}$ and $k\in \{1,2,3\}\setminus \{i,j\}$, $n\in \{4,5,6\}\setminus \{l,m\}$,

\noindent
(ix) $(i\ l)(j\ m)(k\ n)$ for $\{i,j,k\}=\{1,2,3\}$ and $\{l,m,n\}=\{4,5,6\}$. 

\noindent
Hence we have the desired representatives. We also omit the details. 
\end{proof}

\begin{Remark}\label{wu_sign}
{\rm 
We can decide the sign of ${\mathcal L}(f)$ in Lemma \ref{auto_simon} as follows. Let $D_{2}(G)$ be the union of $s\times t$ where $(s,t)$ varies over all pairs of disjoint edges of $G$. It is known that $D_{2}(G)$ is homotopy equivalent to $C_{2}(G)$, equivariantly with respect to the action $\iota$ \cite[Proposition 1.4]{taniyama95}. Moreover, $D_{2}(K_{5})$ and $D_{2}(K_{3,3})$ are homeomorphic to the closed connected orientable surfaces of genus $6$ and $4$, respectively \cite{sarkaria91}. Then, for an automorphism $\sigma$ of $G$, ${\mathcal L}(f\circ \sigma)={\mathcal L}(f)$ (resp. ${\mathcal L}(f\circ \sigma)=-{\mathcal L}(f)$) if and only if $(\sigma\times \sigma)|_{D_{2}(G)}$ is orientation-preserving (resp. orientation-reversing). Note that if two automorphisms $\sigma$ and $\tau$ are conjugate and $(\sigma\times\sigma)|_{D_{2}(G)}$ is orientation-preserving (resp. orientation-reversing), then $(\tau\times\tau)|_{D_{2}(G)}$ is also orientation-preserving (resp. orientation-reversing). As a consequence, $(\sigma\times \sigma)|_{D_{2}(K_{5})}$ is orientation-reversing if and only if $\sigma$ is conjugate to $(1~2~3~4)$, $(1~2)$ or $(1~2)(3~4~5)$, and $(\sigma\times\sigma)|_{D_{2}(K_{3,3})}$ is orientation-reversing if and only if $\sigma$ is conjugate to $(1~4~2~5)(3~6)$, $(1~2)$ or $(1~2)(4~5~6)$ by a direct observation how the oriented 2-cells of $D_{2}(G)$ are mapped (or we may check that by a direct calculation of the Simon invariants under two different numberings of the vertices that differ by an automorphism of $G$). In only these cases, it may happen that ${\mathcal L}(f\circ \sigma)=-{\mathcal L}(f)$.
} 
\end{Remark}

\section{Proofs of main theorems} 

For the proofs of Theorems \ref{K5} and \ref{K33}, we need two results which have been proved as consequences of $3$-manifold topology. The following is a direct corollary of \cite[Theorem 2]{soma98}. 

\begin{Theorem}\label{soma}  
Let $G=K_{5}$ or $K_{3,3}$, $\sigma$ an automorphism of $G$, and $f$ a $\sigma$-symmetric spatial embedding of $G$. Then there exist a spatial embedding $g$ of $G$ which is homologous to $f$ and a periodic self-homeomorphism $\varphi$ of ${\mathbb S}^{3}$ such that $g\circ\sigma=\varphi\circ g$. 
\end{Theorem}

The following is well known as Smith's theorem. 

\begin{Theorem}\label{smith} {\rm (\cite{smith39})} 
Let $\varphi$ be an orientation-reversing periodic self-homeomorphism of ${\mathbb S}^{3}$. Then the fixed point set ${\rm Fix}(\varphi)$ is homeomorphic to ${\mathbb S}^{0}$ or ${\mathbb S}^{2}$. 
\end{Theorem}

In the following proofs we denote the image of the vertex with label $i$ under the spatial embedding also by $i$, so long as no confusion arises. 

\begin{proof}[Proof of Theorem \ref{K5}.] 
(3) and (4) Let $f_{m,\pm}$ be the spatial embeddings of $K_5$ as in Fig. \ref{4.1}. By a calculation we have ${\mathcal L}(f_{m,\pm})=4m \pm 1$, which may be any odd number for a suitable choice of $m$ and $\pm 1$. A $\pi/2$ rotation around the axis through the vertex $5$, the middle point of the edge $\overline{1\ 3}$ and the middle point of the edge $\overline{2\ 4}$ maps $f_{m,\pm}(K_5)$ onto its mirror image. Therefore $f_{m,\pm}$ is $(1\ 2\ 3\ 4)$-symmetric. Thus $((1\ 2\ 3\ 4),n)$ is realizable for any odd integer $n$. A $\pi$ rotation maps $f_{m,\pm}(K_5)$ onto itself and hence $f_{m,\pm}$ is $(1\ 3)(2\ 4)$-symmetric. Therefore $((1\ 3)(2\ 4),n)$ is realizable for any odd integer $n$. Since $(1\ 2)(3\ 4)$ is conjugate to $(1\ 3)(2\ 4)$, Proposition \ref{conjugate1} implies $((1\ 2)(3\ 4),n)$ is realizable for any odd integer $n$. 

\begin{figure}[htbp]
      \begin{center}
\scalebox{0.425}{\includegraphics*{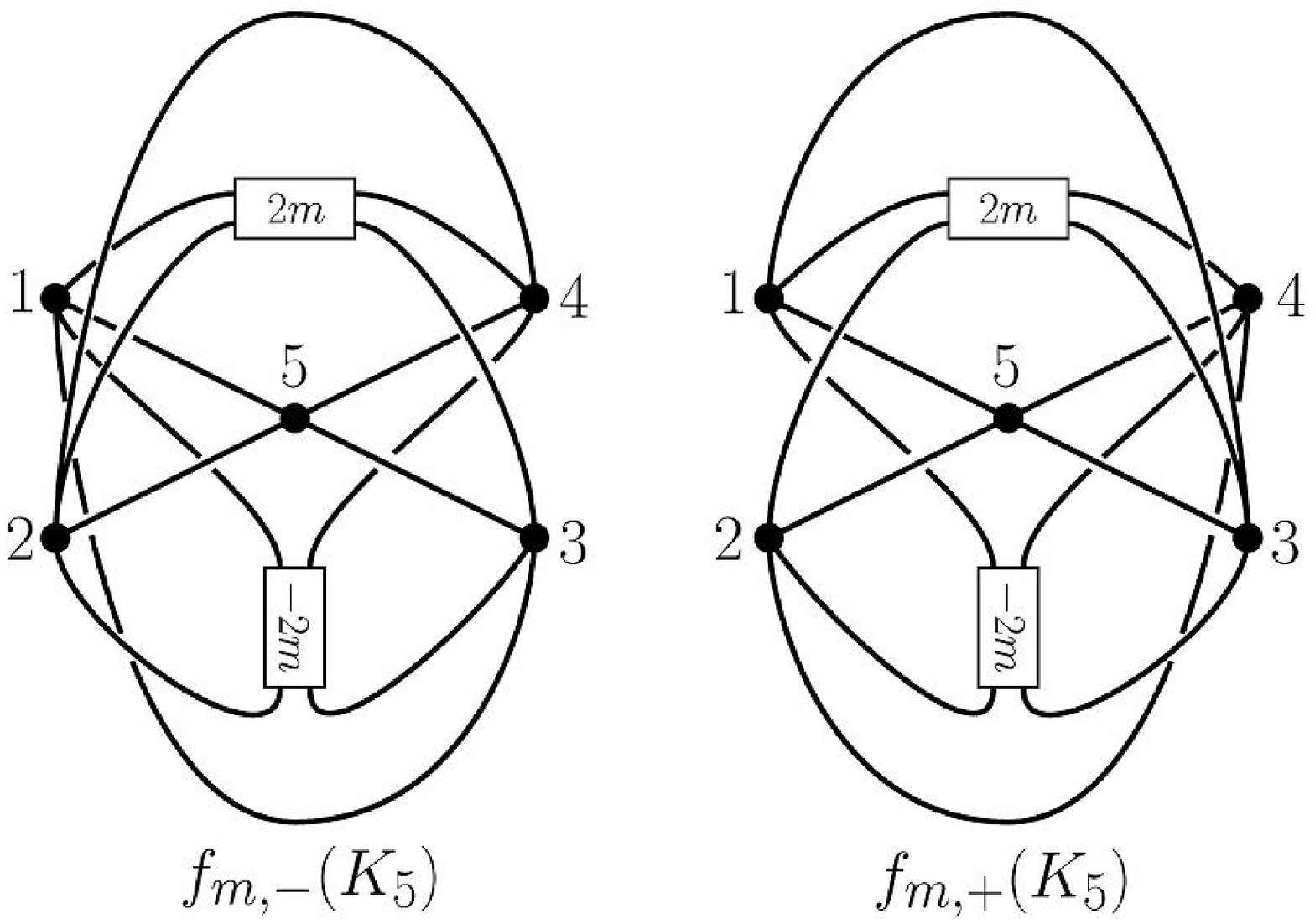}}
      \end{center}
   \caption{}
  \label{4.1}
\end{figure} 

\noindent
(2) Let $f_{m}$ be the spatial embedding of $K_5$ as in Fig. \ref{4.2}. Note that ${\mathcal L}(f_{m})=2m+1$, which may be any odd number. Observe that $f_{m}(K_5)$ is contained in an unknotted M\"{o}bius strip in ${\mathbb S}^{3}$ containing $m$ full twists. Then by considering a suitable ambient isotopy of ${\mathbb S}^{3}$ that preserves the M\"{o}bius strip setwise we conclude that $f_{m}$ is $(1\ 2\ 3\ 4\ 5)$-symmetric.

\begin{figure}[htbp]
      \begin{center}
\scalebox{0.3}{\includegraphics*{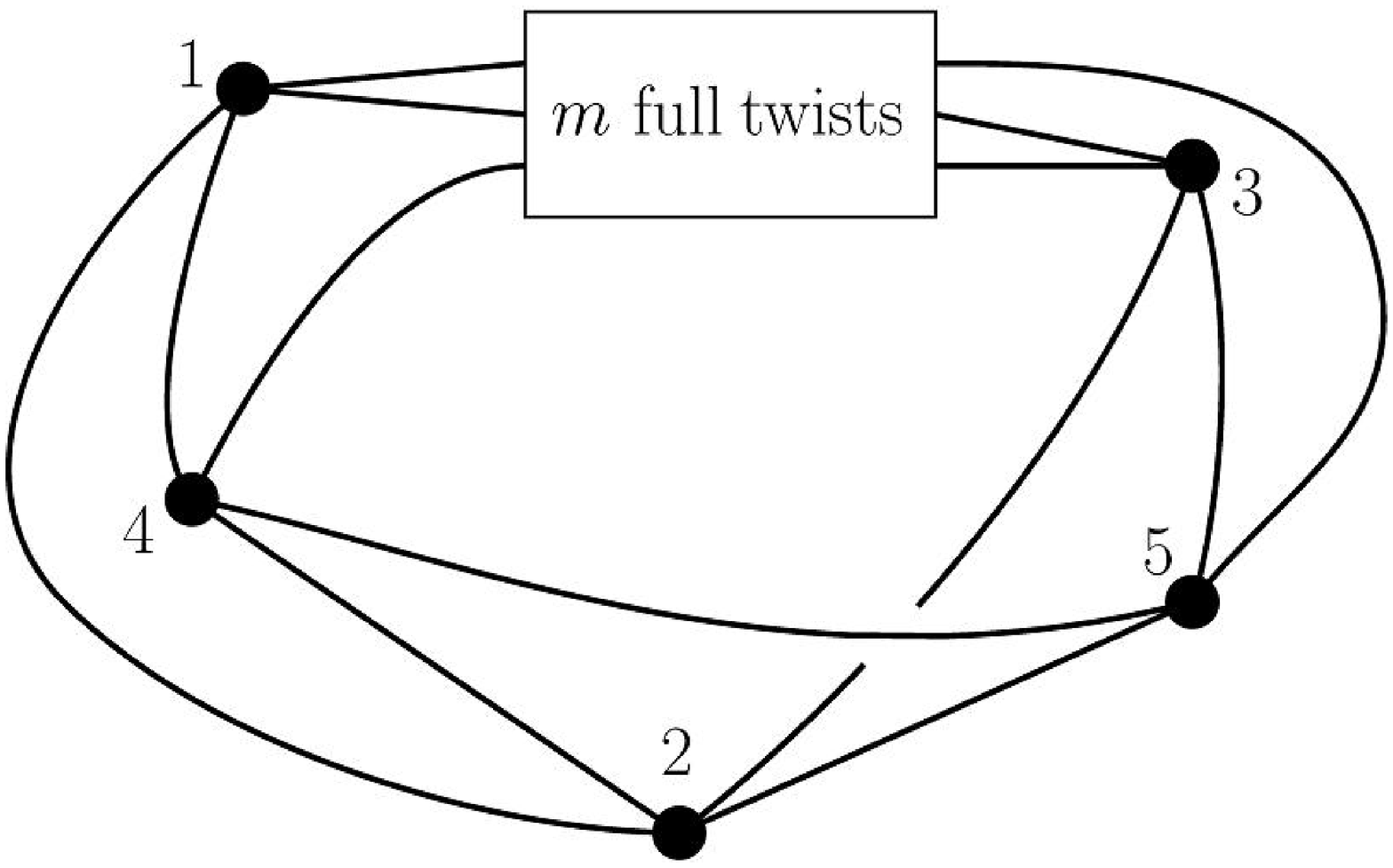}}
      \end{center}
   \caption{}
  \label{4.2}
\end{figure} 

\noindent
(1) (``if'' part) Let $f_{m,\pm}$ be the spatial embedding of $K_5$ as in Fig. \ref{4.3}. Note that ${\mathcal L}(f_{m,\pm})=6m\pm1$. A $2\pi/3$ rotation around the axis through the vertices $4$ and $5$ maps $f_{m,\pm}(K_5)$ onto itself and so $f_{m,\pm}$ is $(1\ 2\ 3)$-symmetric.

\begin{figure}[htbp]
      \begin{center}
\scalebox{0.45}{\includegraphics*{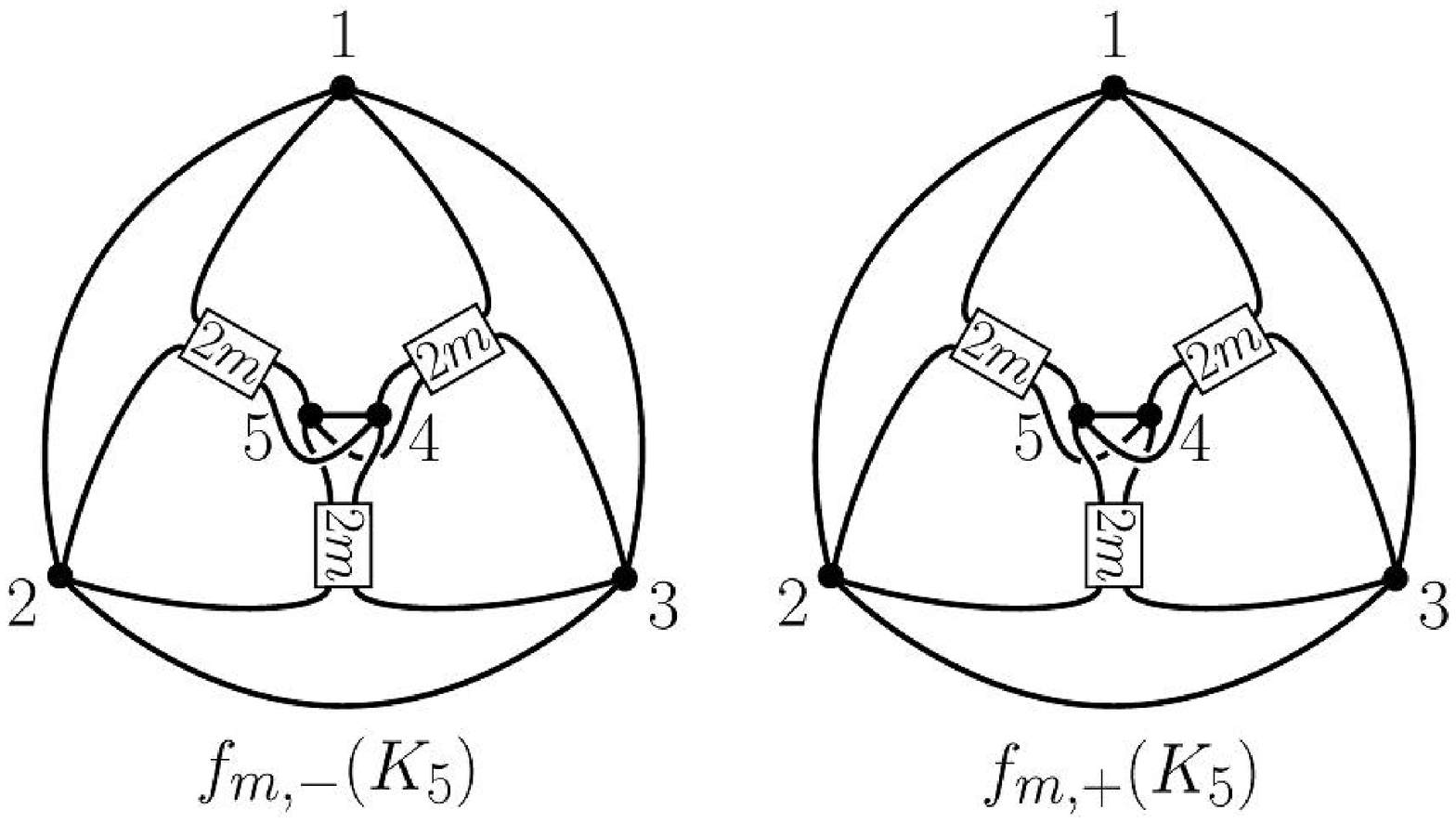}}
      \end{center}
   \caption{}
  \label{4.3}
\end{figure} 

\noindent
(5) (``if'' part) and (6) (``if'' part) Let $f_{\pm}$ be the spatial embedding of $K_5$ as in Fig. \ref{4.4}. Note that ${\mathcal L}(f_{\pm})=\pm 1$. By considering the reflection of ${\mathbb S}^{3}$ with respect to the $2$-sphere containing the cycle $[3\ 4\ 5]$, we see that $f_{\pm}$ is $(1\ 2)$-symmetric. Then by composing this reflection and a $2\pi/3$ rotation around the axis through the vertices $1$ and $2$, we find that $f_{\pm}$ is $(1\ 2)(3\ 4\ 5)$-symmetric.

\begin{figure}[htbp]
      \begin{center}
\scalebox{0.325}{\includegraphics*{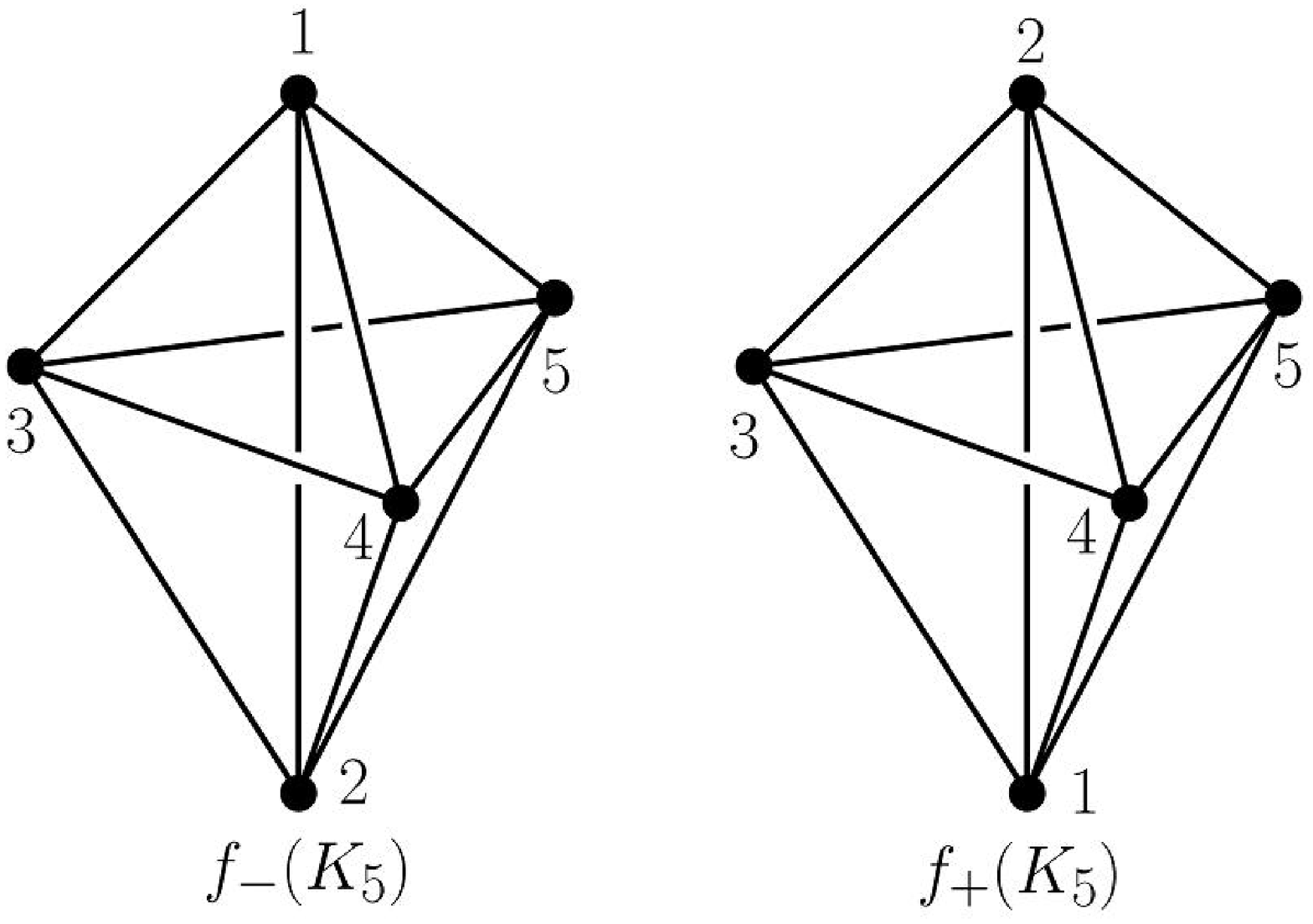}}
      \end{center}
   \caption{}
  \label{4.4}
\end{figure} 

\noindent
(1) (``only if'' part) Suppose that $f$ is a $(1\ 2\ 3)$-symmetric spatial embedding of $K_{5}$. We will show that $\alpha_{\omega}(f)$ is a multiple of $3$. Then by Lemma \ref{mod6} we have the result. There exist twelve $5$-cycles and fifteen $4$-cycles in $K_5$. By the permutation $(1\ 2\ 3)$ they are divided into the following nine orbits: 
\begin{eqnarray*}
& [1\ 2\ 3\ 4\ 5]\mapsto[2\ 3\ 1\ 4\ 5] \mapsto [3\ 1\ 2\ 4\ 5]\mapsto[1\ 2\ 3\ 4\ 5] &\\
& [1\ 2\ 3\ 5\ 4]\mapsto[2\ 3\ 1\ 5\ 4] \mapsto [3\ 1\ 2\ 5\ 4]\mapsto[1\ 2\ 3\ 5\ 4] &\\
& [1\ 4\ 2\ 5\ 3]\mapsto[2\ 4\ 3\ 5\ 1]\mapsto[3\ 4\ 1\ 5\ 2]\mapsto[1\ 4\ 2\ 5\ 3] & \\
& [1\ 5\ 2\ 4\ 3]\mapsto[2\ 5\ 3\ 4\ 1]\mapsto[3\ 5\ 1\ 4\ 2]\mapsto[1\ 5\ 2\ 4\ 3] & \\
& [1\ 2\ 3\ 4]\mapsto[2\ 3\ 1\ 4] \mapsto [3\ 1\ 2\ 4]\mapsto[1\ 2\ 3\ 4] &\\
& [1\ 2\ 3\ 5]\mapsto[2\ 3\ 1\ 5] \mapsto [3\ 1\ 2\ 5]\mapsto[1\ 2\ 3\ 5] &\\
& [1\ 2\ 4\ 5]\mapsto[2\ 3\ 4\ 5] \mapsto [3\ 1\ 4\ 5]\mapsto[1\ 2\ 4\ 5] &\\
& [1\ 2\ 5\ 4]\mapsto[2\ 3\ 5\ 4] \mapsto [3\ 1\ 5\ 4]\mapsto[1\ 2\ 5\ 4] &\\
& [1\ 4\ 2\ 5]\mapsto[2\ 4\ 3\ 5] \mapsto [3\ 4\ 1\ 5]\mapsto[1\ 4\ 2\ 5] &
\end{eqnarray*}
Note that each orbit contains exactly three cycles and these cycles are mapped onto the same knot under $f$. Therefore they have the same $a_2$ and we have the desired conclusion.

\noindent
(5) (``only if'' part) Suppose that $f$ is a $(1\ 2)$-symmetric spatial embedding of $K_{5}$. Then by Theorem \ref{soma} there exist a $(1\ 2)$-symmetric spatial embedding $g$ of $K_{5}$ which is homologous to $f$ and a periodic self-homeomorphism $\varphi$ of ${\mathbb S}^{3}$ such that $g\circ(1\ 2)=\varphi\circ g$. As pointed out in Remark \ref{wu_sign}, it follows that ${\mathcal L}(g\circ(1\ 2))=-{\mathcal L}(g)$. Thus ${\mathcal L}(\varphi\circ g)={\mathcal L}(g\circ(1\ 2))=-{\mathcal L}(f)$, so $\varphi$ is orientation-reversing. Then by Theorem \ref{smith}, ${\rm Fix}(\varphi)$ is homeomorphic to ${\mathbb S}^{0}$ or ${\mathbb S}^{2}$. Since ${\rm Fix}(\varphi)$ contains at least three points $g(3)$, $g(4)$ and $g(5)$, it is homeomorphic to ${\mathbb S}^{2}$. Then ${\rm Fix}(\varphi)\cap g(K_5)$ is the union of the $3$-cycle $g([3\ 4\ 5])$ and the middle point of the edge $g(\overline{1\ 2})$. By considering the regular projection which is almost perpendicular to ${\rm Fix}(\varphi)$ and observing that the signs of crossings by the dotted lines in Fig. \ref{4.5} are not counted in the Simon invariant we deduce that ${\mathcal L}(g)=\pm 1$. Since $f$ and $g$ are homologous it follows that ${\mathcal L}(f)={\mathcal L}(g)=\pm 1$.

\begin{figure}[htbp]
      \begin{center}
\scalebox{0.5}{\includegraphics*{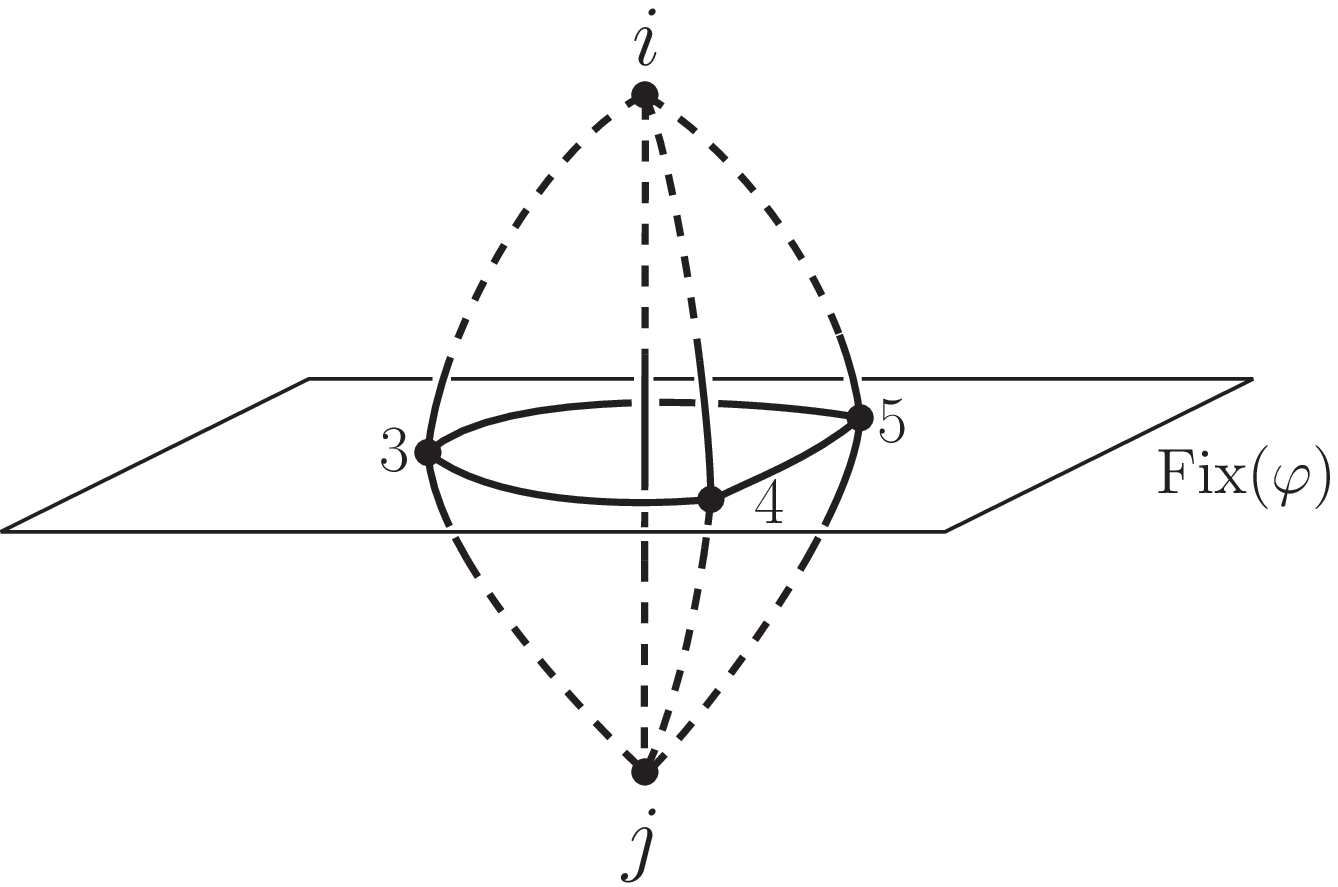}}
      \end{center}
   \caption{}
  \label{4.5}
\end{figure} 

\noindent
(6) (``only if'' part) Note that the cube of $(1\ 2)(3\ 4\ 5)$ equals $(1\ 2)$. Therefore $(1\ 2)(3\ 4\ 5)$-symmetric spatial embeddings of $K_{5}$ are $(1\ 2)$-symmetric. Then we have the result by (5) (``only if'' part). 
\end{proof}

\begin{proof}[Proof of Theorem \ref{K33}.] 
(2) and (6) Let $f_{m,\pm}$ be the spatial embeddings of $K_{3,3}$ as in Fig. \ref{4.6}. Note that ${\mathcal L}(f_{m,\pm})=4m\pm 1$, which may be any odd number. A $\pi/2$ rotation maps $f_{m,\pm}(K_{3,3})$ onto its mirror image. Therefore $f_{m,\pm}$ is $(1\ 4\ 2\ 5)(3\ 6)$-symmetric. Also, a $\pi$ rotation maps $f_{m,\pm}(K_{3,3})$ onto itself and hence $f_{m,\pm}$ is $(1\ 2)(4\ 5)$-symmetric. 

\begin{figure}[htbp]
      \begin{center}
\scalebox{0.425}{\includegraphics*{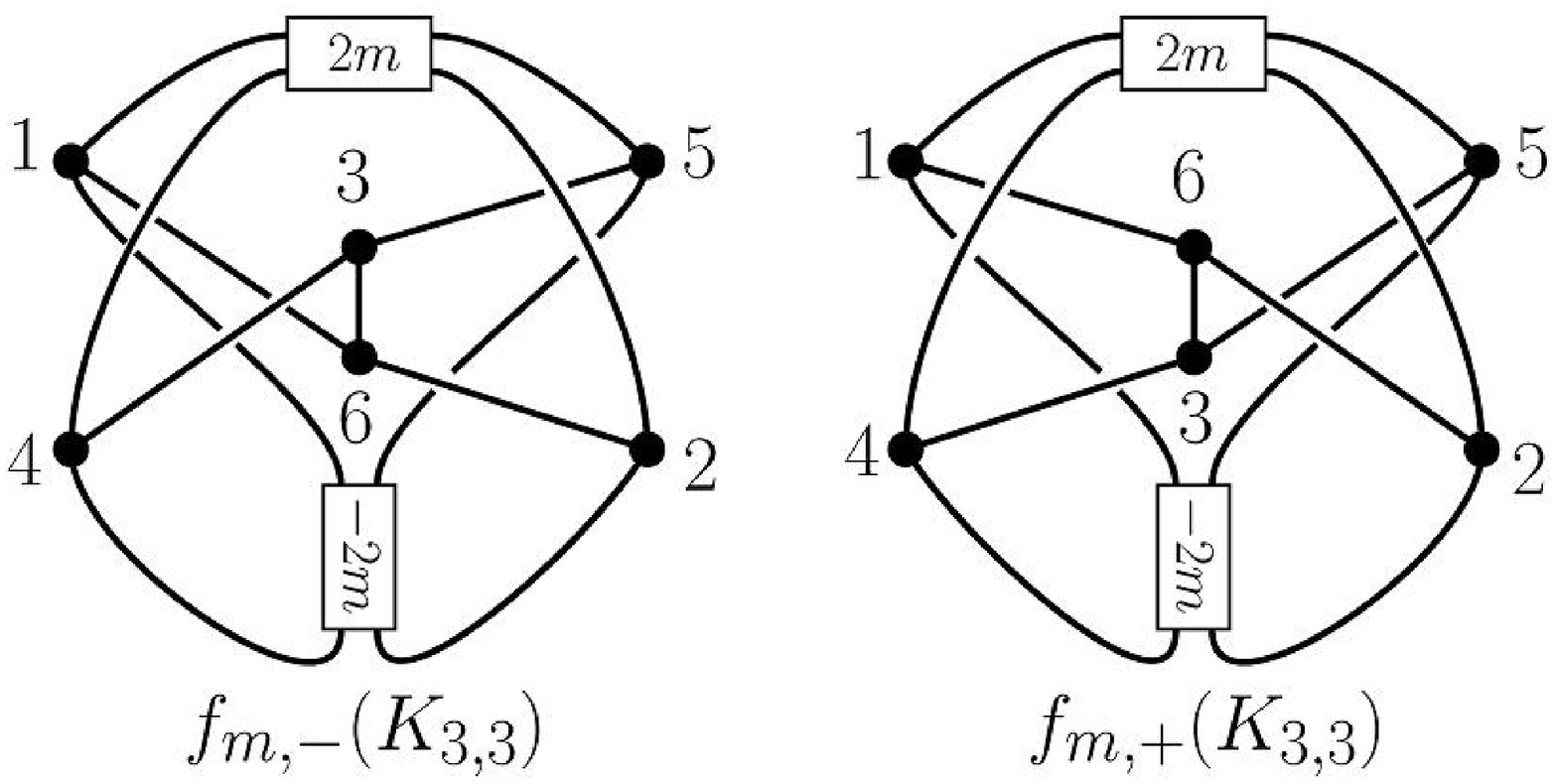}}
      \end{center}
   \caption{}
  \label{4.6}
\end{figure} 

\noindent
(3) (4) (5) Let $f_{m}$ be the spatial embedding of $K_{3,3}$ as in Fig. \ref{4.7}. Note that ${\mathcal L}(f_{m})=2m+1$. Observe that $f_{m}(K_{3,3})$ is contained in an unknotted M\"{o}bius strip in ${\mathbb S}^{3}$ containing $2m+1$ half twists. Then by considering a suitable ambient isotopy of ${\mathbb S}^{3}$ that preserves the M\"{o}bius strip setwise we see that $f_{m}$ is $(1\ 4\ 2\ 5\ 3\ 6)$-symmetric. Note that the square of $(1\ 4\ 2\ 5\ 3\ 6)$ is $(1\ 2\ 3)(4\ 5\ 6)$. Therefore $f_{m}$ is $(1\ 2\ 3)(4\ 5\ 6)$-symmetric. It is also easy to see that $f_{m}$ is $(1\ 5)(2\ 6)(3\ 4)$-symmetric. Since $(1\ 5)(2\ 6)(3\ 4)$ is conjugate to $(1\ 4)(2\ 5)(3\ 6)$ we have the result.

\begin{figure}[htbp]
      \begin{center}
\scalebox{0.25}{\includegraphics*{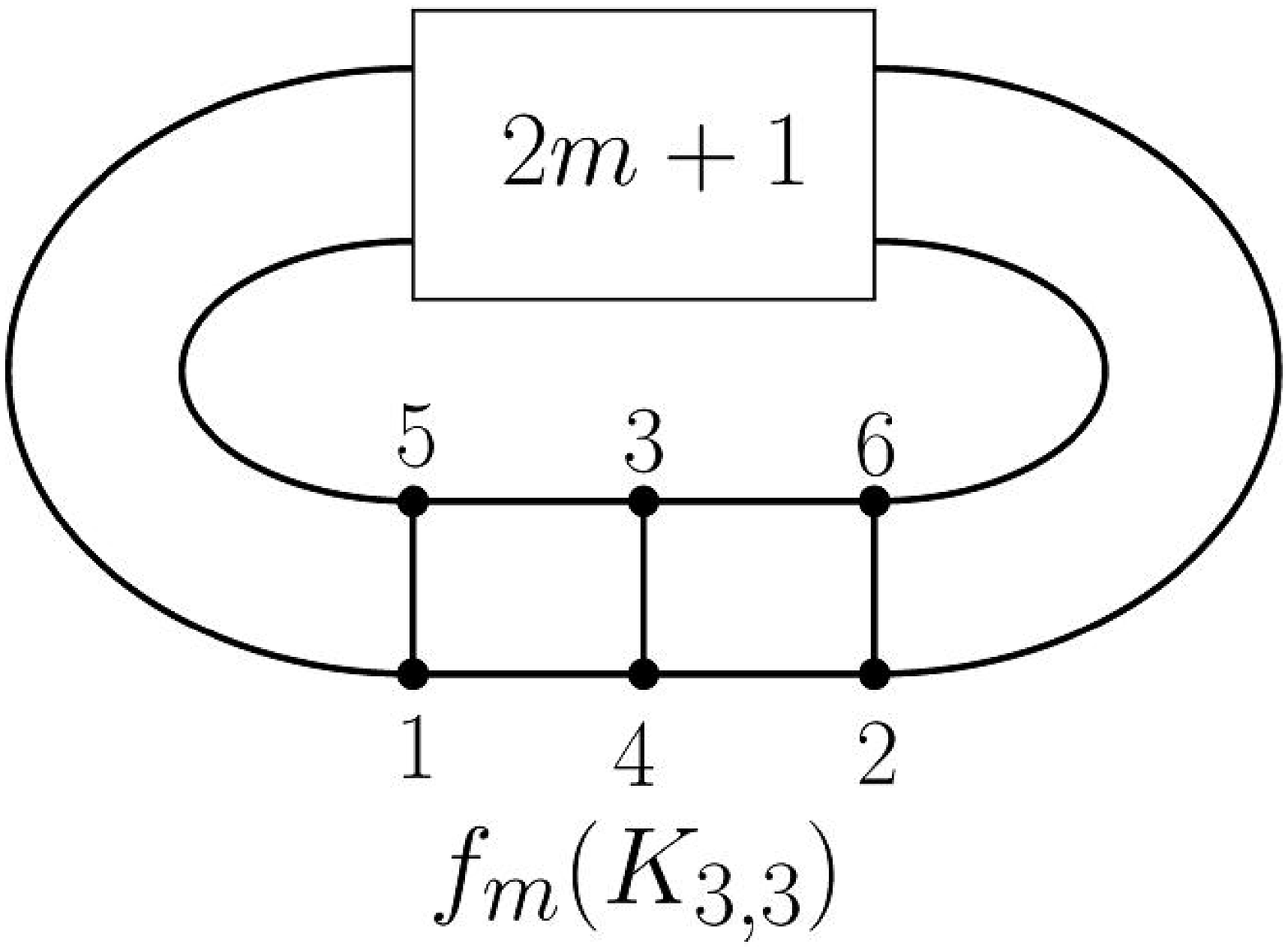}}
      \end{center}
   \caption{}
  \label{4.7}
\end{figure} 

\noindent
(1) (``if'' part) Let $f_{m,\pm}$ be the spatial embedding of $K_{3,3}$ as in Fig. \ref{4.8}. Note that ${\mathcal L}(f_{m,\pm})=6m\pm 1$. A $2\pi/3$ rotation around the axis through the vertices $5$, $4$ and $6$ maps $f_{m,\pm}(K_{3,3})$ onto itself, so $f_{m,\pm}$ is $(1\ 2\ 3)$-symmetric.

\begin{figure}[htbp]
      \begin{center}
\scalebox{0.4}{\includegraphics*{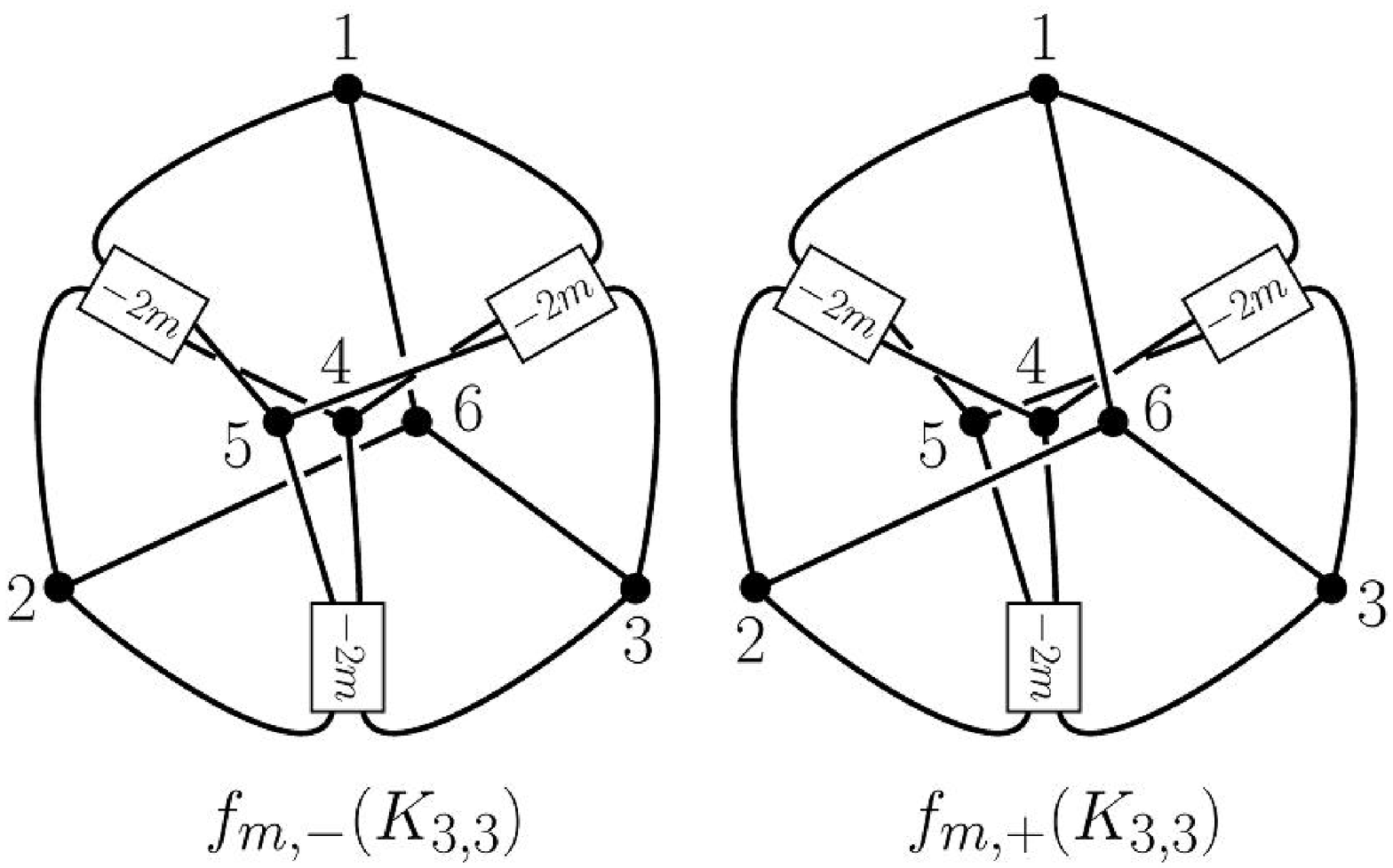}}
      \end{center}
   \caption{}
  \label{4.8}
\end{figure} 

\noindent
(7) (``if'' part) and (8) (``if'' part) Let $f_{\pm}$ be the spatial embedding of $K_{3,3}$ as in Fig. \ref{4.9}. Note that ${\mathcal L}(f_{\pm})=\pm 1$. By considering the reflection of ${\mathbb S}^{3}$ with respect to the $2$-sphere containing the vertices $3$, $4$, $5$ and $6$, we find that $f_{\pm}$ is $(1\ 2)$-symmetric. Then by composing this reflection and a $2\pi/3$ rotation around the axis through the vertices $1$, $2$ and $3$, we conclude that $f_{\pm}$ is $(1\ 2)(4\ 5\ 6)$-symmetric.

\begin{figure}[htbp]
      \begin{center}
\scalebox{0.325}{\includegraphics*{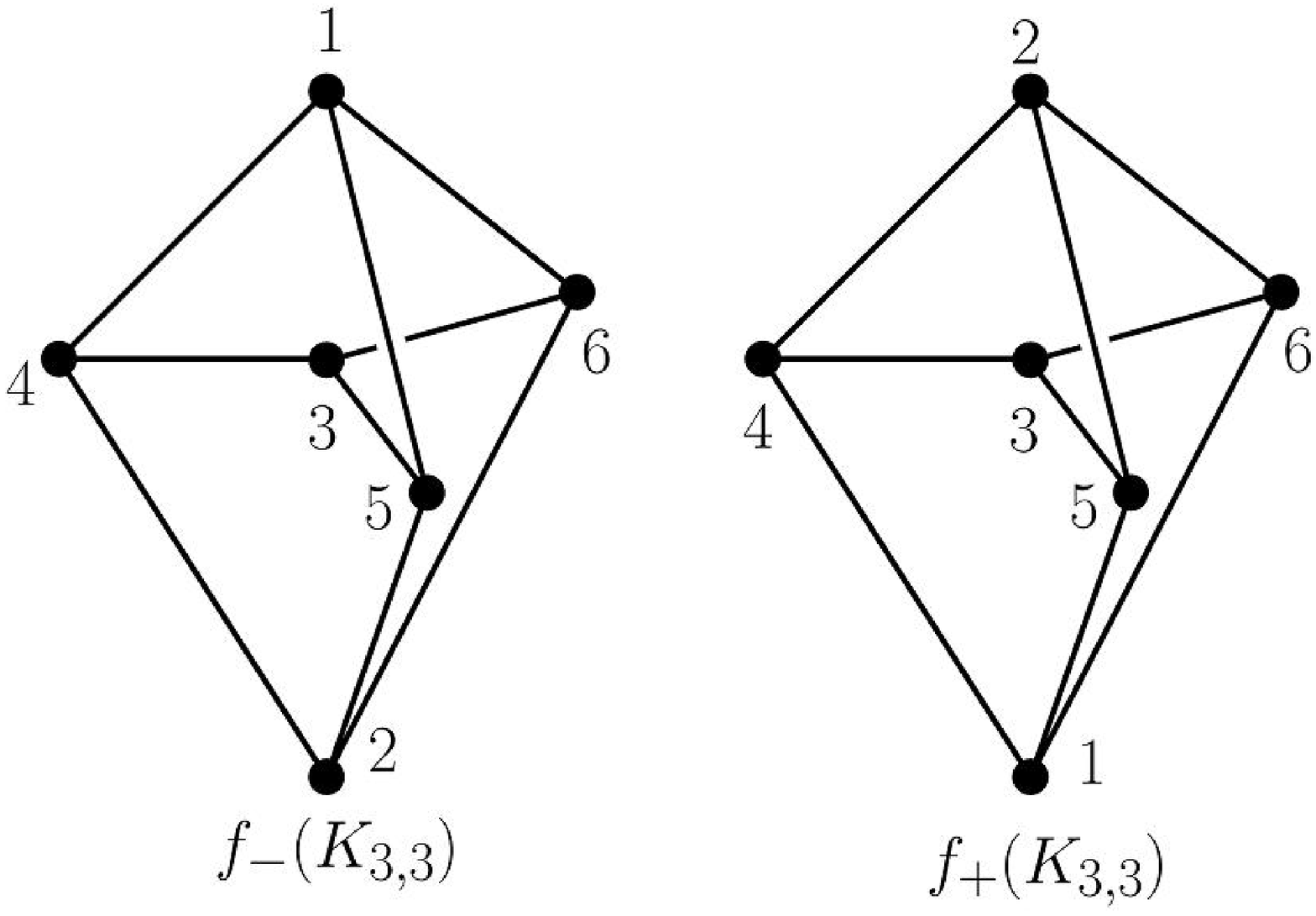}}
      \end{center}
   \caption{}
  \label{4.9}
\end{figure} 

\noindent
(1) (``only if'' part) Suppose that $f$ is a $(1\ 2\ 3)$-symmetric spatial embedding of $K_{3,3}$. We will show that $\alpha_{\omega}(f)$ is a multiple of $3$. Then by Lemma \ref{mod6} we have the result. There exist six $6$-cycles and nine $4$-cycles in $K_{3,3}$. By the permutation $(1\ 2\ 3)$ they are divided into the following five orbits: 

\begin{eqnarray*}
& [1\ 4\ 2\ 5\ 3\ 6]\mapsto[2\ 4\ 3\ 5\ 1\ 6]\mapsto[3\ 4\ 1\ 5\ 2\ 6]\mapsto[1\ 4\ 2\ 5\ 3\ 6] & \\
& [1\ 6\ 2\ 5\ 3\ 4]\mapsto[2\ 6\ 3\ 5\ 1\ 4]\mapsto[3\ 6\ 1\ 5\ 2\ 4]\mapsto[1\ 6\ 2\ 5\ 3\ 4] & \\
& [1\ 4\ 2\ 5]\mapsto[2\ 4\ 3\ 5]\mapsto[3\ 4\ 1\ 5]\mapsto[1\ 4\ 2\ 5] & \\
& [1\ 5\ 2\ 6]\mapsto[2\ 5\ 3\ 6]\mapsto[3\ 5\ 1\ 6]\mapsto[1\ 5\ 2\ 6] & \\
& [1\ 6\ 2\ 4]\mapsto[2\ 6\ 3\ 4]\mapsto[3\ 6\ 1\ 4]\mapsto[1\ 6\ 2\ 4] & 
\end{eqnarray*}
Note that each orbit contains exactly three cycles and they are mapped onto the same knot under $f$. Therefore they have the same $a_2$ and we have the desired conclusion.

\noindent
(7) (``only if'' part) Suppose that $f$ is a $(1\ 2)$-symmetric spatial embedding of $K_{3,3}$. Then by Theorem \ref{soma} there exist a $(1\ 2)$-symmetric spatial embedding $g$ of $K_{3,3}$ which is homologous to $f$ and a periodic self-homeomorphism $\varphi$ of ${\mathbb S}^{3}$ such that $g\circ(1\ 2)=\varphi\circ g$. By the same reason as in the proof of Theorem \ref{K5} (5) (``only if'' part), $\varphi$ is orientation-reversing. Then by Theorem \ref{smith}, ${\rm Fix}(\varphi)$ is homeomorphic to ${\mathbb S}^{0}$ or ${\mathbb S}^{2}$. Since ${\rm Fix}(\varphi)$ contains at least four points $g(3)$, $g(4)$, $g(5)$ and $g(6)$, it is homeomorphic to ${\mathbb S}^{2}$. Then ${\rm Fix}(\varphi)\cap g(K_{3,3})$ is the induced subgraph of the vertices $g(3)$, $g(4)$, $g(5)$ and $g(6)$. By considering the regular projection which is almost perpendicular to ${\rm Fix}(\varphi)$ and by the same reason as in the proof of Theorem \ref{K5} (5) (``only if'' part), we deduce that ${\mathcal L}(g)=\pm 1$ (see Fig. \ref{4.10}). Since $f$ and $g$ are homologous it follows that ${\mathcal L}(f)={\mathcal L}(g)=\pm 1$.

\begin{figure}[htbp]
      \begin{center}
\scalebox{0.5}{\includegraphics*{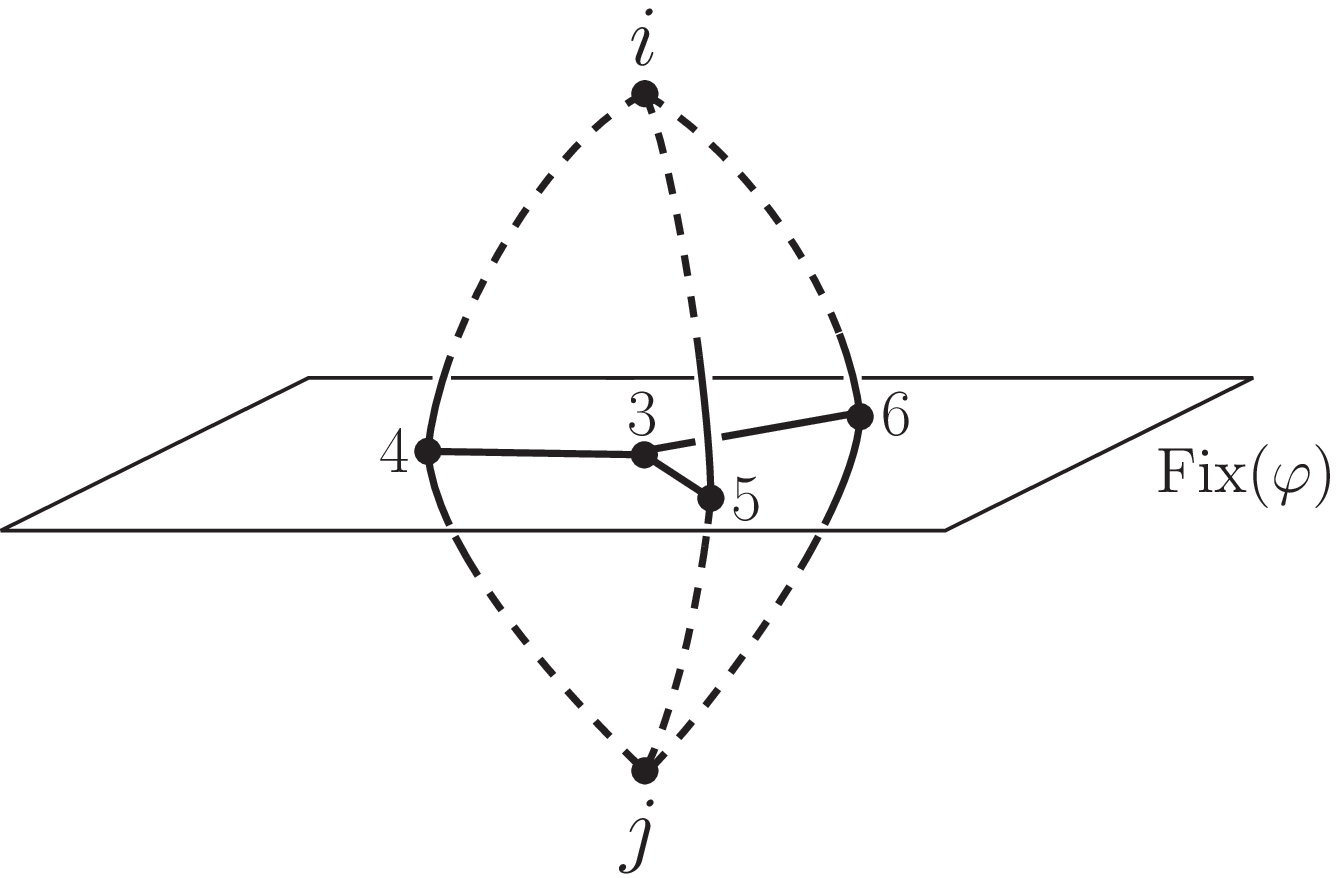}}
      \end{center}
   \caption{}
  \label{4.10}
\end{figure} 

\noindent
(8) (``only if'' part) Note that the cube of $(1\ 2)(4\ 5\ 6)$ equals $(1\ 2)$. Therefore $(1\ 2)(4\ 5\ 6)$-symmetric spatial embeddings of $K_{3,3}$ are $(1\ 2)$-symmetric. Then we have the result by (7) (``only if'' part). 
\end{proof}

\begin{proof}[Proof of Theorem \ref{achiral}.] 
Let $G$ be $K_5$ or $K_{3,3}$. For a spatial embedding $f$ of $G$ and a self-homeomorphism $\varphi$ of ${\mathbb S}^3$ we have ${\mathcal L}(\varphi\circ f)={\mathcal L}(f)$ if $\varphi$ is orientation-preserving and ${\mathcal L}(\varphi\circ f)=-{\mathcal L}(f)$ if $\varphi$ is orientation-reversing. By combining this fact with Lemma \ref{auto_simon} and Remark \ref{wu_sign}, we have the results. 
\end{proof}

\begin{Remark}
{\rm 
(1) All symmetries shown by various examples in this paper are realized by periodic self-homeomorphisms of ${\mathbb S}^3$. Hence all symmetric spatial graphs in this paper are rigidly symmetric.

\noindent
(2) There are alternative proofs of the ``only if'' parts of Theorems \ref{K5} (1) and \ref{K33} (1), based on Theorem \ref{soma} and the Smith conjecture \cite{morgan-bass79}, just as the proofs of Theorems \ref{K5} (5) and (6) and \ref{K33} (7) and (8) given above are based on Theorems \ref{soma} and \ref{smith}. However, as we said before, the latter theorems are consequences of deep results of $3$-manifold topology. The proof of the ``only if'' parts of Theorems \ref{K5} (1) and \ref{K33} (1) given above are elementary. To find elementary proofs of the ``only if'' parts of Theorems \ref{K5} (5) and (6) and \ref{K33} (7) and (8) is an open problem.
}
\end{Remark}

{\normalsize
}

\end{document}